\documentclass[12pt,leqno]{amsart}
\usepackage{times}
\usepackage{amssymb}
\usepackage[dvips]{graphicx}
\usepackage{amsmath,amssymb,txfonts}
\usepackage{amssymb}
\usepackage[unicode=true]{hyperref}
\usepackage{amsxtra}
\usepackage{amsmath}
\usepackage{txfonts}
\usepackage{amsmath,amssymb,txfonts}
\usepackage{amssymb}
\usepackage{amsxtra}
\usepackage{amsmath}
\usepackage{txfonts}
\usepackage[cp1252]{inputenc}
\usepackage{mathrsfs}
\usepackage{graphicx}
\usepackage[active]{srcltx}

\allowdisplaybreaks
\newtheorem{thm}{Theorem}

\newtheorem{cor}[thm]{Corollary}
\newtheorem{prop}[thm]{Proposition}
\newtheorem{rem}{Remark}
\newtheorem{defn}{Definition}

\input{tcilatex}

\begin{document}
\title[The $\infty$-Besov Capacity Problem]{The $\infty$-Besov Capacity
Problem}
\author{M. Milman}
\address{Instituto Argentino de Matematica, Buenos Aires, Argentina}
\email{mario.milman@gmail.com}
\author{J. Xiao}
\address{Department of Mathematics and Statistics, Memorial University, St.
John's, NL A1C5S7, Canada}
\email{jxiao@mun.ca}
\thanks{MM was partially supported by a grant from the Simons Foundation
(\#207929 to Mario Milman), JX is in part supported by NSERC of Canada.}

\begin{abstract}
A theory of $\infty$-Besov capacities is developed and several applications
are provided. In particular, we solve an open problem in the theory of
limits of the $\infty$-Besov semi-norms, we obtain new restriction-extension
inequalities, and we characterize the point-wise multipliers acting on the $%
\infty$-Besov spaces.
\end{abstract}

\maketitle

%    General info

%\dedicatory{In memory of Adriano M. Garsia 1928-2010}

%\dedicatory{Dedicated to Adriano M. Garsia who surely appreciated a simplified approach}

%\tableofcontents

\section{Introduction}

\label{s1}

In this paper (cf. Section $2$ below) we introduce a new\footnote{%
The theory builds from previous works (cf. \cite{A, AX, X1, X2, XY, XZ, W, C}%
).} theory of capacities and perimeters associated to the Besov spaces $%
\Lambda _{\alpha }^{p,q},$with parameters $(\alpha ,p,q)\in (0,1)\times
\lbrack 1,\infty ]\times \lbrack 1,\infty ],$ with particular emphasis on
the case $q=\infty $ (the $\infty -$Besov spaces)$.$ Our theory has
interesting applications: In Section $3$ we apply it to characterize the
restrictions and extensions of the $\infty $-Besov functions, and in Section 
$4$, we provide a characterization of the point-wise multipliers for the $%
(\alpha ,p,\infty )$-Besov spaces. There are, of course, many other
interesting connections. For example, we mention that the corresponding
spaces of traces are naturally linked to the theory of function spaces based
on outer measures that was recently developed in \cite{DT} (cf. Section $3$%
). Moreover, and somewhat surprisingly, the $\infty $-Besov spaces can be
embedded in the Campanato spaces (cf. Remark \ref{r21b} below).

Another interesting application occurs when dealing with an open end point
problem in the theory of limits of Besov or fractional Sobolev norms. We
shall now develop this application in some detail, as we believe it offers a
nice application and an introduction to some of underlying issues dealt with
in this paper.

The limiting inequalities we intend to extend originated in applications to
PDEs, and were considered by a number of authors (cf. \cite{BBM}, \cite{BBM1}%
, \cite{MAS}, \cite{MAS1}, \cite{KL}). For example,
Bourgain-Brezis-Mironescu \cite{BBM1} show that 
\begin{equation*}
\lim_{\alpha \rightarrow 1}(1-\alpha )\|f\|^p_{\dot{\Lambda}%
^{p,p}_\alpha}=\lim_{\alpha \rightarrow 1}(1-\alpha)\int_{\mathbb{R}^n}\int_{%
\mathbb{R}^n}\frac{\left\vert f(x)-f(y)\right\vert ^{p}}{\left\vert
x-y\right\vert ^{n+\alpha p}}\,dxdy=\left(\int_{\mathbb{S}%
^{n-1}}|\cos\theta|^p\,d\sigma\right)\left\Vert \nabla f\right\Vert
_{L^{p}}^{p},
\end{equation*}%
where $p\in \lbrack 1,\infty )$, $\mathbb{S}^{n-1}$ is the unit sphere of $%
\mathbb{R}^{n\ge 1}$, $\theta$ is the angle derivation from the vertical,
and $d\sigma$ is the standard surface area measure. In \cite{MAS, MAS1}
Maz'ya-Shaposhnikova show 
\begin{equation*}
\lim_{\alpha \rightarrow 0}\alpha\|f\|^p_{\dot{\Lambda}^{p,p}_\alpha}=\lim_{%
\alpha\rightarrow 0}\alpha \int_{\mathbb{R}^{n}}\int_{\mathbb{R}^{n}} \frac{%
\left\vert f(x)-f(y)\right\vert ^{p}}{\left\vert x-y\right\vert ^{n+\alpha p}%
}dxdy=2p^{-1}\sigma _{n-1}\left\Vert f\right\Vert _{L^{p}}^{p},
\end{equation*}%
where $p\in \lbrack 1,\infty ),$ and $\sigma _{n-1}$ denotes the surface
measure of $\mathbb{S}^{n-1}.$ A related inequality obtained in \cite{BBM1},
and sharpened in \cite{MAS}, can be formulated as%
\begin{equation*}
\left\Vert f\right\Vert _{L^{\frac{pn}{n-\alpha p}}}\leq \left(c(n,p)
\alpha(1-\alpha )(n-\alpha p)^{1-p}\right) ^{\frac{1}{p}}\left\Vert
f\right\Vert _{\dot{\Lambda}_\alpha^{p,p}},
\end{equation*}%
where $(\alpha ,p)\in (0,1)\times \lbrack 1,n/\alpha)$ and $c(n,p)$ is a
constant depending only on $n$ and $p$. To derive each of these results
required a new understanding about fractional norms\footnote{%
See also Remark \ref{r31}(ii) below.}.

From a more general point of view, the Besov spaces $\Lambda _{\alpha
}^{p,q} $ can be also obtained by real interpolation and, as it turns out,
the limiting formulae above can be also understood in the more general
setting of interpolation theory (cf. \cite{M}). In particular, this point of
view led us to extend these limiting theorems to higher order norms (cf. 
\cite{KMX}). Further results in this direction have been also obtained by
Triebel \cite{tri}.

A natural question that has remained open is the characterization of the
limits of the homogeneous Besov norms $\Vert \cdot \Vert _{\dot{\Lambda}%
_{\alpha }^{p,q}}$ that correspond to the choices $p=\infty $ or $q=\infty .$
As it is well known, the $\infty $-Besov spaces $\Lambda^{p,\infty}_\alpha$
are connected to the Sobolev spaces. Indeed, the if $(\alpha ,p)\in
\{1\}\times (1,\infty )$ then $f$ is in the first-order Sobolev $p$-space $%
W^{1,p}$ if and only if $f\in L^{p}$ and $\sup_{h\in \mathbb{R}%
^{n}}|h|^{-1}\Vert \Delta _{h}f\Vert _{L^{p}}<\infty $ (cf. \cite[Theorem
2.16]{Zi}), moreover, if $\alpha =1=p,$ then $f$ is of bounded variation on $%
\mathbb{R}^{n}$, i.e. $f\in BV$, if and only if $f\in L^{1}$ and $\sup_{h\in 
\mathbb{R}^{n}}|h|^{-1}\Vert \Delta _{h}f\Vert _{L^{1}}<\infty $ (cf. \cite[%
p. 245]{CDDD}).

Returning to the limiting theorems above, let us now show how our theory of
capacities can be used to add a new end point result to the
Bourgain-Brezis-Mironescu-Maz'ya-Shaposhnikova formulae. Let $\Vert \cdot
\Vert _{BV}$ denote the standard BV-norm, and let 
\begin{equation*}
P_{1,1,\infty }(\{x\in \mathbb{R}^{n}:f(x)>t\})=\Vert 1_{\{f>t\}}(\cdot
)\Vert_{\dot{\Lambda}_{\alpha }^{1,\infty }}
\end{equation*}%
(cf. Section $2$). Suppose that $f\in BV(\mathbb{R}^{n})$ and let $\alpha
\in (0,1).$ Since $f\in L^{1}(\mathbb{R}^{n}),$ we have (cf. Proposition \ref%
{p23}(1) and (3) below for more details) that, for each $h\in $ $\mathbb{R}%
^{n},$ 
\begin{align*}
\left\vert h\right\vert ^{-1}\Vert \Delta _{h}f\Vert _{L^{1}}& \leq
\int_{0}^{\infty }|h|^{-1}\Vert 1_{\{f>t\}}(\cdot +h)-1_{\{f>t\}}(\cdot
)\Vert _{L^{1}}\,dt \\
& \leq \int_{0}^{\infty }|h|^{-1}\left\vert h\right\vert P_{1,1,\infty
}(\{x\in \mathbb{R}^{n}:f(x)>t\})\,dt \\
& \approx \left\Vert f\right\Vert _{BV}.
\end{align*}%
Let us write%
\begin{equation*}
\Vert f\Vert _{\dot{\Lambda}_{\alpha }^{\alpha ,\infty }}\leq \sup_{|h|\leq
1}|h|^{-\alpha }\Vert \Delta _{h}f\Vert _{L^{1}}+\sup_{|h|>1}|h|^{-\alpha
}\Vert \Delta _{h}f\Vert _{L^{1}}.
\end{equation*}%
Let $\epsilon >0.$ Then we can find non-zero $h_{1}:=h_{1}(\alpha )\in 
\mathbb{R}^{n},$ with $\left\vert h_{1}\right\vert \leq 1,$ such that 
\begin{equation*}
\Vert f\Vert _{\dot{\Lambda}_{\alpha }^{1,\infty }}<|h_{1}|^{-\alpha }\Vert
\Delta _{h_{1}}f\Vert _{L^{1}}+\sup_{|h|>1}|h|^{-\alpha }\Vert \Delta
_{h}f\Vert _{L^{1}}+\epsilon .  \label{alto}
\end{equation*}%
The first term in the last inequality can be estimated by 
\begin{eqnarray*}
|h_{1}|^{-\alpha }\Vert \Delta _{h_{1}}f\Vert _{L^{1}} &\leq
&|h_{1}|^{1-\alpha }|h_{1}|^{-1}\Vert \Delta _{h_{1}}f\Vert _{L^{1}} \\
&\leq&|h_{1}|^{-1}\Vert \Delta _{h_{1}}f\Vert _{L^{1}} \\
&\lesssim &\left\Vert f\right\Vert _{BV}.
\end{eqnarray*}%
Moreover, using the mean value theorem and noting that the function $%
t\mapsto t^{\alpha }\ln t$ is bounded on $[0,1],$ we see that $\lim_{\alpha
\rightarrow 1}t^{\alpha }=t,$ uniformly on $[0,1],$ thus 
\begin{equation*}
\limsup_{\alpha \rightarrow 1}\sup_{|h|>1}|h|^{-\alpha }\Vert \Delta
_{h}f\Vert _{L^{1}}\leq \sup_{|h|>1}|h|^{-1}\Vert \Delta _{h}f\Vert
_{L^{1}}\lesssim\|f\|_{BV}.
\end{equation*}%
Collecting estimates, we see that 
\begin{equation*}
\limsup_{\alpha \rightarrow 1}\Vert f\Vert _{\dot{\Lambda}_{\alpha
}^{1,\infty }}\lesssim \left\Vert f\right\Vert _{BV}+\epsilon .
\end{equation*}%
Consequently, letting $\epsilon \rightarrow 0,$ we obtain 
\begin{equation*}
\limsup_{\alpha \rightarrow 1}\Vert f\Vert _{\dot{\Lambda}_{\alpha
}^{1,\infty }}\lesssim \Vert f\Vert _{BV}.
\end{equation*}

\smallskip \noindent\textit{Notation.} In the above and below, $A\approx B$
means $A\lesssim B\lesssim A$; while $A\lesssim B$ stands for $A\leq cB$ for
a constant $c>0$.

\section{$\infty$-Besov spaces and capacities}

\label{s2}

\subsection{Besov spaces}

\label{s21} The Besov spaces $\Lambda _{\alpha }^{p,q}(\mathbb{R}^{n})\equiv
\Lambda _{\alpha }^{p,q}$ that we shall consider in this paper will be
defined in terms of difference operators (cf. e.g. \cite{Si, KuJF}). Let $%
h\in \mathbb{R}^{n}$, the difference operator, $\Delta _{h}$, acting on
functions $f$ defined on $\mathbb{R}^{n}$, is given by, 
\begin{equation*}
\Delta _{h}f(x)=f(x+h)-f(x)\quad \forall \quad x\in \mathbb{R}^{n}.
\end{equation*}

\begin{defn}
\label{d21}The ${\Lambda }_{\alpha }^{p,q}$ spaces are defined according to
the values of the parameters $(\alpha ,p,q)$ as follows:

\begin{itemize}
\item[(i)] If $(\alpha ,p,q)\in (0,1)\times \lbrack 1,\infty]\times \lbrack
1,\infty ),$ then ${\Lambda }_{\alpha }^{p,q}$ is the class of all $L^{p}$%
-functions $f$ such that 
\begin{equation*}
\Vert f\Vert _{\dot{\Lambda}_{\alpha }^{p,q}}=\left( \int_{\mathbb{R}%
^{n}}\Vert \Delta _{h}f\Vert _{L^{p}}^{q}{|h|^{-(n+\alpha q)}}\,dh\right) ^{%
\frac{1}{q}}<\infty .
\end{equation*}

\item[(ii)] If $(\alpha ,p,q)\in (0,1)\times \lbrack 1,\infty]\times
\{\infty \}$ then ${\Lambda }_{\alpha }^{p,q}$ is the class of all $L^{p}$%
-functions $f$ such that 
\begin{equation*}
\Vert f\Vert _{\dot{\Lambda}_{\alpha }^{p,q}}=\sup_{h\in \mathbb{R}%
^{n}}\Vert \Delta _{h}f\Vert _{L^{p}}|h|^{-\alpha }<\infty .
\end{equation*}
\end{itemize}
\end{defn}

For perspective, two interesting and important remarks on Definition \ref%
{d21} are given below to show further connections of the $\infty-$Besov
spaces with some classical function spaces.

\begin{rem}
\label{r21a} Two comments on Definition \ref{d21}(i) are in order.

\begin{itemize}
\item[(i)] It is interesting to observe that there is a natural Leibniz rule
associated to $\Vert \cdot \Vert _{\dot{\Lambda}_{\alpha }^{p,q}}$, which is
connected with a family of $BMO$-based Besov spaces extending the case $%
p=\infty$ of Definition \ref{d21}(i) due to $L^\infty\subset BMO$; see also
Remark \ref{r21b}(ii) . Let 
\begin{equation*}
\Vert f\Vert _{\dot{\Lambda}^{BMO,q}_\alpha}= 
\begin{cases}
\left( \int_{\mathbb{R}^{n}}\Vert \Delta _{h}f\Vert _{BMO}^{q}{%
|h|^{-(n+\alpha q)}}\,dh\right) ^{\frac{1}{q}}\ \ \hbox{as}\ \ q\in \lbrack
1,\infty ); \\ 
\sup_{h\in \mathbb{R}^{n}}\Vert \Delta _{h}f\Vert _{BMO}|h|^{-\alpha }\ \ %
\hbox{as}\ \ q=\infty .%
\end{cases}%
\end{equation*}

Then (cf. \cite{KT} and more recently \cite[(1.22)]{Milman}), if $p\in
(1,\infty ]$ and $h\in \mathbb{R}^{n}$, we have 
\begin{align*}
\Vert \Delta _{h}(fg)\Vert _{L^{p}}& =\Vert f(x+h)g(x+h)-f(x)g(x)\Vert
_{L^{p}} \\
& =\Vert \Delta _{h}f(\cdot )g(\cdot +h)+f(\cdot )\Delta _{h}g(\cdot )\Vert
_{L^{p}} \\
& \leq \Vert \Delta _{h}f(\cdot )g(\cdot +h)\Vert _{L^{p}}+\Vert f\Delta
_{h}g\Vert _{L^{p}} \\
& \leq \left\Vert \Delta _{h}f\right\Vert _{L^{p}}\Vert g(\cdot +h)\Vert
_{BMO}+\left\Vert \Delta _{h}f\right\Vert _{BMO}\Vert g(\cdot +h)\Vert
_{L^{p}}+\left\Vert f\right\Vert _{L^{p}}\Vert \Delta _{h}g\Vert
_{BMO}+\Vert f\Vert _{BMO}\Vert \Delta _{h}g\Vert _{L^{p}} \\
& =\left\Vert \Delta _{h}f\right\Vert _{L^{p}}\Vert g\Vert _{BMO}+\left\Vert
\Delta _{h}f\right\Vert _{BMO}\Vert g\Vert _{L^{p}}+\left\Vert f\right\Vert
_{L^{p}}\Vert \Delta _{h}g\Vert _{BMO}+\Vert f\Vert _{BMO}\Vert \Delta
_{h}g\Vert _{L^{p}}
\end{align*}%
where we have used Minkowski's inequality and the translation-invariance of $%
\Vert \cdot \Vert _{L^{p}}$ and $\Vert \cdot \Vert _{BMO}$. Consequently, 
\begin{align*}
\Vert fg\Vert _{\dot{\Lambda}_{\alpha }^{p,q}}& \lesssim \Vert f\Vert _{\dot{%
\Lambda}_{\alpha }^{p,q}}\Vert g\Vert _{BMO}+\Vert g\Vert _{\dot{\Lambda}%
_{\alpha }^{p,q}}\Vert f\Vert _{BMO}+\Vert f\Vert _{L^{p}}\Vert g\Vert _{%
\dot{\Lambda}_{\alpha }^{BMO,q}}+\Vert f\Vert _{\dot{\Lambda}_{\alpha
}^{BMO,q}}\Vert g\Vert _{L^{p}} \\
& \lesssim \big(\Vert f\Vert _{\dot{\Lambda}_{\alpha }^{p,q}}+\Vert f\Vert
_{L^{p}}\big)\Big(\Vert g\Vert _{BMO}+\Vert g\Vert _{\dot{\Lambda}_{\alpha
}^{BMO,q}}\Big)+\big(\Vert g\Vert _{\dot{\Lambda}_{\alpha }^{p,q}}+\Vert
g\Vert _{L^{p}}\big)\Big(\Vert f\Vert _{BMO}+\Vert f\Vert _{\dot{\Lambda}%
_{\alpha }^{BMO,q}}\Big).
\end{align*}

\item[(ii)] Let 
\begin{equation*}
I_{\alpha }g=\mathcal{F}^{-1}(|\zeta |^{-\alpha }\hat{g}(\zeta ))
\end{equation*}%
be the $\alpha $-Riesz potential of $g$ defined via the Fourier transform $%
\hat{g}=\mathcal{F}g$ and the inverse Fourier transform $\mathcal{F}^{-1}$.
Moreover, let 
\begin{equation*}
\Vert f\Vert _{I_{\alpha }(BMO)}=\Vert g\Vert _{BMO}\quad \hbox{if}\quad
f=I_{\alpha }g.
\end{equation*}%
Let $1<q_{1}<2<q_{2}<\infty .$ Then, the following implications hold for $%
\dot{\Lambda}_{\alpha }^{\infty ,q}$(cf. \cite[Theorem 3.4]{St}): 
\begin{align*}
\Vert f\Vert _{\dot{\Lambda}_{\alpha }^{\infty ,1}}<\infty & \Rightarrow
\Vert f\Vert _{\dot{\Lambda}_{\alpha }^{\infty ,q_{1}}}<\infty  \\
& \Rightarrow \Vert f\Vert _{\dot{\Lambda}_{\alpha }^{\infty ,2}}<\infty  \\
& \Rightarrow \Vert f\Vert _{I_{\alpha }(BMO)}<\infty \ \hbox{or}\ \Vert
f\Vert _{\dot{\Lambda}_{\alpha }^{\infty ,q_{2}}}<\infty  \\
& \Rightarrow \Vert f\Vert _{\dot{\Lambda}_{\alpha }^{\infty ,\infty
}}<\infty .
\end{align*}
\end{itemize}
\end{rem}

\begin{rem}
\label{r21b} Two comments on Definition \ref{d21}(ii) are in order:

\begin{itemize}
\item[(i)] From \cite[Theorem 2.16]{Zi} and \cite[p. 245]{CDDD} it follows
that $f$ is in the first-order Sobolev $p$-space $W^{1,p}$ if and only if 
\begin{equation*}
f\in L^p\quad\&\quad \sup_{h\in \mathbb{R}^{n}}|h|^{-1}\Vert \Delta
_{h}f\Vert _{L^{p}}<\infty ,
\end{equation*}
and $f$ is of bounded variation on $\mathbb{R}^{n}$, i.e. $f\in BV$, if and
only if 
\begin{equation*}
f\in L^{1}\quad\&\quad \sup_{h\in \mathbb{R}^{n}}|h|^{-1}\Vert \Delta
_{h}f\Vert _{L^{1}}<\infty.
\end{equation*}
So, for $(\alpha,p)\in(0,1)\times[1,\infty)$, the $\infty$-Besov space $%
\Lambda_\alpha^{p,\infty}$ can be treated as the fractional extensions of $%
W^{1,p}$ and $BV$.

\item[(ii)] Although the inclusions 
\begin{equation*}
\dot{\Lambda}_{\alpha }^{p,q}\subset \dot{\Lambda}_{\alpha }^{p,\infty
},(\alpha ,p,q)\in (0,1)\times \lbrack 1,\infty )\times \lbrack 1,\infty ),
\end{equation*}%
are well-known, it is quite surprising that for $(\alpha ,p)\in (0,1)\times
\lbrack 1,\infty )$ the homogeneous $\infty $-Besov space $\dot{\Lambda}%
_{\alpha }^{p,\infty }$ embeds in the family of Campanato spaces. Towards a
proof of this fact this let $\delta >0$ and $(\alpha ,p)\in (0,1)\times
\lbrack 1,\infty )$; then if $f\in \dot{\Lambda}_{\alpha }^{p,\infty }$ we
have 
\begin{equation*}
\int_{|h|<\delta }\Vert \Delta _{h}f\Vert _{L^{p}}^{p}\,dh\lesssim \Vert
f\Vert _{\dot{\Lambda}_{\alpha }^{p,\infty }}^{p}\delta ^{\alpha p+n}.
\end{equation*}%
Consequently, 
\begin{equation*}
\int_{|x-x_{0}|<\frac{\delta }{2}}\int_{|y-x_{0}|<\frac{\delta }{2}%
}|f(x)-f(y)|^{p}\,dydx\lesssim \Vert f\Vert _{\dot{\Lambda}_{\alpha
}^{p,\infty }}^{p}\delta ^{\alpha p+n}.
\end{equation*}%
Let $B(x_{0},r)$ be the Euclidean ball with center $x_{0}$, radius $r,$ and
Lebesgue measure $|B(x_{0},r)|$. Let $\fint_{E}$ stand for the integral mean
of $f$ over $E\subset \mathbb{R}^{n},$ with respect to the Lebesgue measure $%
dx$ or $dy$. Using Jensen's inequality we readily obtain 
\begin{equation*}
\int_{B(x_{0},\frac{\delta }{2})}\int_{B(x_{0},\frac{\delta }{2}%
)}|f(x)-f(y)|^{p}\,dydx\geq \Big|B(x_{0},\frac{\delta }{2})\Big|%
\int_{B(x_{0},\frac{\delta }{2})}\Big|f(x)-\fint_{B(x_{0},\frac{\delta }{2}%
)}f\Big|^{p}\,dx,
\end{equation*}%
thus we see that 
\begin{equation*}
\Vert f\Vert _{p,\alpha }=\sup_{(x_{0},r)\in \mathbb{R}^{n}\times (0,\infty
)}r^{-\alpha }\left\Vert f-\fint_{B(x_{0},r)}f\right\Vert
_{L^{p}(B(x_{0},r))}\lesssim \Vert f\Vert _{\dot{\Lambda}_{\alpha
}^{p,\infty }}.
\end{equation*}%
We note the following consequences of the previous discussion:

\begin{itemize}
\item If $\alpha p<n$, then $\Vert f\Vert _{\dot{\Lambda}_{\alpha
}^{p,\infty }}$ $<\infty ,$ implies that $f$ belongs to the $(-\alpha +\frac{%
n}{p})$-Campanato class $\mathcal{L}^{p,\alpha p}$, i.e., $\Vert f\Vert _{%
\mathcal{L}^{p,\alpha p}}=\Vert f\Vert _{p,\alpha }<\infty $; cf. \cite[p.67]%
{Gia}.

\item If $\alpha p=n$, then $\Vert f\Vert _{\dot{\Lambda}_{\alpha
}^{p,\infty }}$ $<\infty $ implies that $f$ belongs to the class $BMO$ of
functions with bounded mean oscillation, i.e., $\Vert f\Vert _{BMO}=\Vert
f\Vert _{p,\frac{n}{p}}<\infty $; cf. \cite{JN}.

\item If $\alpha p>n$, then $\Vert f\Vert _{\dot{\Lambda}_{\alpha
}^{p,\infty }}<\infty $ implies that $f$ is $(\alpha -\frac{n}{p})$-Hö%
lder-continuous, i.e., $\Vert f\Vert _{\dot{\Lambda}_{\alpha -\frac{n}{p}%
}^{\infty ,\infty }}\approx \|f\|_{p,\alpha}<\infty $; cf. \cite[p.70]{Gia}.
\end{itemize}
\end{itemize}
\end{rem}

\subsection{$\infty$-Besov capacities}

\label{s22}

Motivated by \cite[p.27]{HKM} and Remark \ref{r21b}(ii), we introduce the
following definition.

\begin{defn}
\label{d22} Let $C(\mathbb{R}^n)$ be the class of all continuous functions
on $\mathbb{R}^n$. Then the (homogeneous) $\infty$-Besov capacity of a set $%
E\subset\mathbb{R}^n$ is defined by 
\begin{equation*}
C_{\alpha,p,q}(E)=%
\begin{cases}
\inf\big\{\|f\|^p_{\dot{\Lambda}^{p,\infty}_\alpha}:\ f\in\mathsf{A}%
_{\alpha,p,\infty}(E)\big\}\quad\hbox{as}\quad (\alpha,p,q)\in (0,1)\times[%
1,\infty]\times\{\infty\}; \\ 
\inf\big\{\|f\|^q_{\dot{\Lambda}^{\infty,q}_\alpha}:\ f\in\mathsf{A}%
_{\alpha,\infty,q}(E)\big\}\quad\hbox{as}\quad (\alpha,p,q)\in
(0,1)\times\{\infty\}\times[1,\infty],%
\end{cases}%
\end{equation*}
where 
\begin{equation*}
\mathsf{A}_{\alpha,p,q}(E)=\Big\{f\in \Lambda_\alpha^{p,q}\cap C(\mathbb{R}%
^n):\ f\ge 1_N\ \ \hbox{on\ a\ neighbourhood\ N\ of}\ E\Big\}
\end{equation*}
and $1_N$ is the indicator of $N$.
\end{defn}

We shall now explore the nature of the $\infty $-Besov capacities.

\begin{prop}
\label{p21} Suppose $(\alpha ,q)\in (0,1)\times \lbrack 1,\infty ].$ Then $%
C_{\alpha ,\infty ,q}(E)=0\ \forall\ E\subset\mathbb{R}^n$.
\end{prop}

\begin{proof}
The result follows immediately from the fact that constant functions belong
to $\Lambda _{\alpha }^{\infty ,q}\cap C(\mathbb{R}^n)$, when $q\in \lbrack
1,\infty ].$
\end{proof}

Consequently, the interesting situation of $C_{\alpha,p,q}(\cdot)$ is the
capacity $C_{\alpha,p,\infty}$ under $(\alpha,p)\in (0,1)\times[1,\infty)$.
Referring to \cite{KM, C}, we have the following basic properties.

\begin{prop}
\label{p22} Let $(\alpha,p)\in (0,1)\times[1,\infty)$. Then:

\begin{enumerate}
\item $C_{\alpha,p,\infty}(\emptyset)=0$.

\item $C_{\alpha ,p,\infty }(E_{1})\leq C_{\alpha ,p,\infty }(E_{2})$
whenever $E_{1}\subseteq E_{2}$.

\item $C_{\alpha ,p,\infty }(E_1\cup E_2)\leq C_{\alpha ,p,\infty
}(E_1)+C_{\alpha ,p,\infty }(E_1)$ provided $E_1,E_2\subset\mathbb{R}^n$.

\item $C_{\alpha,p,\infty}(E)=\inf\{C_{\alpha,p,\infty}(O):\ \hbox{open}\
O\supseteq E\}$.

\item $C_{\alpha ,p,\infty }(\cap _{j=1}^{\infty }K_{j})=\lim_{j\rightarrow
\infty }C_{\alpha ,p,\infty }(K_{j}),$ where the $K_{j}^{\prime }s$ are
compact subsets of $\mathbb{R}^{n},$ with $K_{j}\supseteq K_{j+1}\ \forall \
j=1,2,3,...$.
\end{enumerate}
\end{prop}

\begin{proof}
(1) and (2) follow immediately from the definition of $C_{\alpha ,p,\infty
}(\cdot )$.

(3) For any given $\epsilon >0$, $j=1,2,$ pick $f_{j}\in \mathsf{A}_{\alpha
,p,\infty }(E_{j})$ such that 
\begin{equation*}
\Vert f_{j}\Vert _{\dot{\Lambda}_{\alpha }^{p,\infty }}^{p}<C_{\alpha
,p,\infty }(E_{j})+\epsilon .
\end{equation*}%
Note that 
\begin{equation*}
\begin{cases}
f=\max \{f_{1},f_{2}\}\in \mathsf{A}_{\alpha ,p,\infty }(E_{1}\cup E_{2});
\\ 
\Vert \Delta _{h}f\Vert _{L^{p}}^{p}\leq \Vert \Delta _{h}f_{1}\Vert
_{L^{p}}^{p}+\Vert \Delta _{h}f_{2}\Vert _{L^{p}}^{p}.%
\end{cases}%
\end{equation*}%
Consequently, 
\begin{equation*}
C_{\alpha ,p,\infty }(E_{1}\cup E_{2})\leq \Vert f\Vert _{\dot{\Lambda}%
_{\alpha }^{p,\infty }}^{p}\leq C_{\alpha ,p,\infty }(E_{1})+C_{\alpha
,p,\infty }(E_{2})+2\epsilon .
\end{equation*}%
Letting $\epsilon \rightarrow 0$ the desired result follows.

(4) In view of (2), the verification of (4) will be complete once we prove 
\begin{equation*}
C_{\alpha ,p,\infty }(E)\geq \inf \{C_{\alpha ,p,\infty }(O):\ \hbox{open}\
O\supseteq E\}.
\end{equation*}%
Now, for any $\epsilon >0$ there exists a function $f_{0}\in \mathsf{A}%
_{\alpha ,p,\infty }(E),$ a neighborhood $O$ of $E$ such that $f_0\geq 1$ in 
$O,$ and, moreover, 
\begin{equation*}
\Vert f_{0}\Vert _{\dot{\Lambda}_{\alpha }^{p,\infty }}^{p}<C_{\alpha
,p,\infty }(E)+\epsilon .
\end{equation*}%
But since we also have 
\begin{equation*}
C_{\alpha ,p,\infty }(O)\leq \Vert f_{0}\Vert _{\dot{\Lambda}_{\alpha
}^{p,\infty }}^{p},
\end{equation*}%
the required inequality follows from combining the last two inequalities and
letting $\epsilon $ go to $0$.

(5) Suppose that $K_{1}\supseteq K_{2}\supseteq K_{3}\cdots \supseteq K=\cap
_{j=1}^{\infty }K_{j}$ is a sequence of compact subsets of $\mathbb{R}^{n}$.
Then $\{C_{\alpha ,p,\infty }(K_{j})\}$ is a decreasing numerical sequence
and therefore has a limit as $j\rightarrow \infty $. Let $O$ be any open set
such that $O\supset K$. Since $K$ is compact, there exists an index $j$ such
that $K_{j}\subset O,$ whence, 
\begin{equation*}
\lim_{j\rightarrow \infty }C_{\alpha ,p,\infty }(K_{j})\leq C_{\alpha
,p,\infty }(O).
\end{equation*}%
The last estimate, combined with (2)\&(4), implies 
\begin{equation*}
C_{\alpha ,p,\infty }(K)\leq \lim_{j\rightarrow \infty }C_{\alpha ,p,\infty
}(K_{j})\leq C_{\alpha ,p,\infty }(K).
\end{equation*}
\end{proof}

\begin{rem}
\label{r21} We have been unable to prove that $C_{\alpha,p,\infty}$ is
countably subadditive, but refer the interested reader to \cite{KM0,KM, HK}
for a discussion on the so-called Sobolev capacity and BV capacity on metric
spaces.
\end{rem}

\begin{prop}
\label{p23} Let $(\alpha ,\beta ,p)\in (0,1)\times (0,\infty )\times \lbrack
1,\infty )$. For $E\subset \mathbb{R}^{n}$ let

\begin{itemize}
\item[$\circ$] $E-h=\{x-h: x\in E\}$ be its $\mathbb{R}^n\ni h$-left
translation;

\item[$\circ$] 
\begin{equation*}
P_{\alpha,p,\infty}(E)=\|1_E\|_{\dot{\Lambda}^{p,\infty}_\alpha}=\sup_{h\in%
\mathbb{R}^n}|h|^{-\alpha}\Big(2\big(|E|-|E\cap(E-h)|\big)\Big)^\frac1p
\end{equation*}
be its $(\alpha,p,\infty)$-perimeter;

\item[$\circ$] 
\begin{equation*}
H^{\beta }(E)=\inf \left\{ \sum_{j=1}^{\infty }\left( \frac{\pi ^{\frac{%
\beta }{2}}}{\Gamma (1+\frac{\beta }{2})}\right) r_{j}^{\beta }:\ \
E\subseteq \cup _{j=1}^{\infty }B(x_{j},r_{j})\right\}
\end{equation*}%
be its $\beta $-dimensional Hausdorff capacity with $\Gamma (\gamma
)=\int_{0}^{\infty }e^{-t}t^{\gamma -1}$ and $\gamma \in (0,\infty )$.
\end{itemize}

\begin{enumerate}
\item Suppose $f\in{\Lambda}_{\alpha }^{p,\infty}\cap C(\mathbb{R}^n)$. Then,

\begin{equation*}
\left( t^{p}C_{\alpha ,p,\infty }\big(\{x\in \mathbb{R}^{n}:|f(x)|>t\}\big)%
\right) ^{\frac{1}{p}}\leq \big\||f|\big\|_{\dot{\Lambda}_{\alpha
}^{p,\infty }}\leq \int_{0}^{\infty }P_{\alpha ,p,\infty }\big(\{x\in 
\mathbb{R}^{n}:|f(x)|>t\}\big)\,dt.
\end{equation*}

\item $\big(C_{\alpha ,p,\infty }(E)\big)^{\frac{1}{p}}\leq P_{\alpha
,p,\infty }(E)$.

\item If $0<H^n(E), H^{n-1}(E)<\infty$, then 
\begin{equation*}
P_{\alpha ,p,\infty }(E)<\infty \Leftrightarrow \alpha \leq p^{-1},
\end{equation*}
in other words, 
\begin{equation*}
P_{\alpha ,p,\infty }(E)=\infty \Leftrightarrow \alpha >p^{-1}.
\end{equation*}

\item Let $\mathbb{B}^{n}$ be the unit ball in $\mathbb{R}^{n}.$ For any
Euclidean ball $B(x_{0},r_{0})$ centered at $x_{0}\in \mathbb{R}^{n}$ and
with radius $r_{0}>0,$ 
\begin{equation*}
C_{\alpha ,p,\infty }\big(B(x_{0},r_{0})\big)=r_{0}^{n-\alpha p}C_{\alpha
,p,\infty }\big(\mathbb{B}^{n}\big).
\end{equation*}

\item 
\begin{equation*}
C_{\alpha ,p,\infty }(E)\leq 
\begin{cases}
\left( \frac{\pi ^{\frac{n-\alpha p}{2}}}{\Gamma (1+\frac{n-\alpha p}{2})}%
\right) ^{-1}C_{\alpha ,p,\infty }(\mathbb{B}^{n})H^{n-\alpha p}(E)\quad %
\hbox{as}\quad p\in \lbrack 1,\frac{n}{\alpha }); \\ 
0\quad \hbox{as}\quad p\in \lbrack \frac{n}{\alpha },\infty ).%
\end{cases}%
\end{equation*}
\end{enumerate}
\end{prop}

\begin{proof}
(1) Without loss of generality, we may assume that $f\in \Lambda _{\alpha
}^{p,\infty }\cap C(\mathbb{R}^n)$ is nonnegative. Note that if 
\begin{equation*}
\{f>t\}=\{x\in\mathbb{R}^n: f(x)>t\}
\end{equation*}
(the upper $0<t$-level set of $f$) then $f/t>1$ in $\{f>t\}$. Therefore, 
\begin{equation*}
C_{\alpha ,p,\infty }(\{f>t\})\leq \Vert f/t\Vert _{\dot{\Lambda}_{\alpha
}^{p,\infty }}^{p},
\end{equation*}%
and the desired weak-type estimate follows: 
\begin{equation*}
t^{p}C_{\alpha ,p,\infty }(\{f>t\})\leq \Vert f\Vert _{\dot{\Lambda}_{\alpha
}^{p,\infty }}^{p}\quad \forall \quad t\in (0,\infty ).
\end{equation*}

To verify the remaining inequality in (1) (viewed as a co-area inequality),
we use (cf. \cite{ADM}) 
\begin{equation*}
|\Delta _{h}f(x)|=\int_{0}^{\infty }\big|1_{\{f>t\}}(x+h)-1_{\{f>t\}}(x)\big|%
\,dt,
\end{equation*}%
and Minkowski's inequality, to derive 
\begin{equation*}
|h|^{-\alpha }\Vert \Delta _{h}f\Vert _{L^{p}}\leq \int_{0}^{\infty
}|h|^{-\alpha }\Vert 1_{\{f>t\}}(\cdot +h)-1_{\{f>t\}}(\cdot )\Vert
_{L^{p}}\,dt\leq \int_{0}^{\infty }\Vert 1_{\{f>t\}}\Vert _{\dot{\Lambda}%
_{\alpha }^{p,\infty }}\,dt,
\end{equation*}%
as desired.

(2) For any $f\in\mathsf{A}_{\alpha,p,\infty}(E)$ with $\|f\|_{\dot{\Lambda}%
^{p,\infty}_\alpha}\to\|1_E\|_{\dot{\Lambda}^{p,\infty}_\alpha}$ we use the
monotonicity of $C_{\alpha ,p,\infty}$ to get 
\begin{equation*}
t\big(C_{\alpha ,p,\infty }(E)\big)^{\frac{1}{p}}\leq \left(t^{p}C_{\alpha
,p,\infty }\big(\{x\in \mathbb{R}^{n}:\ f(x)>t\}\big)\right) ^{\frac{1}{p}%
}\leq \Vert f\Vert _{\dot{\Lambda}_{\alpha }^{p,\infty }}\ \ \forall \ \
t\in (0,1),
\end{equation*}%
and, therefore, 
\begin{equation*}
\big(C_{\alpha ,p,\infty }(E)\big)^{\frac{1}{p}}\leq P_{\alpha ,p,\infty
}(E).
\end{equation*}

(3) Suppose $0<H^n(E), H^{n-1}(E)<\infty$. If $P_{\alpha ,p,\infty }(E)$ is
finite, then 
\begin{equation*}
\Vert 1_{E}(x+h)-1_{E}(x)\Vert _{L^{p}}^{p}\leq \big(P_{\alpha ,p,\infty }(E)%
\big)^{p}|h|^{p\alpha }\quad \forall \quad h\in \mathbb{R}^{n}.
\end{equation*}%
A straightforward computation gives 
\begin{equation*}
\Vert 1_{E}(x+h)-1_{E}(x)\Vert _{L^{p}}^{p}=|E|+|E-h|-2|(E-h)\cap
E|=2(|E|-|(E-h)\cap E|)=\Vert 1_{E}-1_{E-h}\Vert _{L^{1}}.
\end{equation*}%
Therefore, for each natural number $k$, we have 
\begin{equation*}
\Vert 1_{E}-1_{E-h}\Vert _{L^{1}}\leq \sum_{j=0}^{k-1}\Big\|1_{E-\frac{j}{k}%
h}-1_{E-\frac{j+1}{k}h}\Big\|_{L^{1}}\leq \big(P_{\alpha ,p,\infty }(E)\big)%
^{p}k^{1-p\alpha }|h|^{p\alpha }.
\end{equation*}%
Now, if $\alpha>p^{-1}$, then letting $k\to\infty$ in the last estimation
gives 
\begin{equation*}
\Vert 1_{E}-1_{E-h}\Vert _{L^{1}}=0\quad \forall \quad h\in \mathbb{R}^{n},
\end{equation*}
and hence 
\begin{equation*}
\lim_{|h|\rightarrow \infty }\Vert 1_{E}-1_{E-h}\Vert _{L^{1}}=0.
\end{equation*}
However, we have 
\begin{equation*}
\lim_{|h|\rightarrow \infty }\Vert 1_{E}-1_{E-h}\Vert _{L^{1}}=2|E|\in
(0,\infty),
\end{equation*}%
thereby reaching a contradiction. Therefore, we must have $\alpha \leq
p^{-1} $. Conversely, if $\alpha \leq p^{-1}$ holds, then an application of 
\cite[Theorem 1]{Sch} to the symmetric difference of $E$ and $E-h$ readily
shows that 
\begin{equation*}
2(|E|-|E\cap (E-h)|)\leq |h|H^{n-1}(E)\quad \forall \quad h\in \mathbb{R}%
^{n}.
\end{equation*}%
Therefore, by the classical isoperimetric inequality 
\begin{equation*}
\left( \frac{H^{n}(E)}{\Big(\frac{\pi ^{\frac{n}{2}}}{\Gamma (1+\frac{n}{2})}%
\Big)}\right) ^{\frac{1}{n}}\leq \left( \frac{H^{n-1}(E)}{\Big(\frac{\pi ^{%
\frac{n-1}{2}}}{\Gamma (1+\frac{n-1}{2})}\Big)}\right) ^{\frac{1}{n-1}},
\end{equation*}%
and the hypothesis $p^{-1}-\alpha \geq 0$, we find 
\begin{align*}
P_{\alpha ,p,\infty }(E)& \leq \max \left\{ \sup_{h\in \mathbb{B}%
^{n}}|h|^{p^{-1}-\alpha }\big(H^{n-1}(E)\big)^{\frac{1}{p}},\ \sup_{h\in 
\mathbb{R}^{n}\setminus \mathbb{B}^{n}}|h|^{-\alpha }\big(2H^{n}(E)\big)^{%
\frac{1}{p}}\right\} \\
& \leq \big(H^{n-1}(E)\big)^{\frac{1}{p}}+\big(2H^{n}(E)\big)^{\frac{1}{p}}
\\
& \leq \big(H^{n-1}(E)\big)^{\frac{1}{p}}+\left( 2\left( \frac{\pi ^{\frac{n%
}{2}}}{\Gamma (1+\frac{n}{2})}\right) \left( \frac{H^{n-1}(E)}{\Big(\frac{%
\pi ^{\frac{n-1}{2}}}{\Gamma (1+\frac{n-1}{2})}\Big)}\right) ^{\frac{n}{n-1}%
}\right) ^{\frac{1}{p}} \\
& <\infty .
\end{align*}

(4) We notice that if $f\in \mathsf{A}_{\alpha ,p,\infty }\big(B(x_{0},r_{0})%
\big)$ and $f_{r_{0}}(x)=f(r_{0}x)$ then 
\begin{equation*}
\Vert f_{r_{0}}\Vert _{\dot{\Lambda}_{\alpha }^{p,\infty
}}^{p}=r_{0}^{\alpha p-n}\Vert f\Vert _{\dot{\Lambda}_{\alpha }^{p,\infty
}}^{p}.
\end{equation*}%
Therefore, by the definition of $C_{\alpha ,p,\infty }$, we the desired
formula follows.

(5) Let $p\in \lbrack 1,\frac{n}{\alpha })$. Using Propositions \ref{p22}%
(2)-(3) \& \ref{p23}(4) we see that if $E\subseteq \cup _{j=1}^{\infty
}B(x_{j},r_{j})$ then 
\begin{equation*}
C_{\alpha ,p,\infty }(E)\leq \sum_{j=1}^{\infty }C_{\alpha ,p,\infty }\big(%
B(x_{j},r_{j})\big)=\sum_{j=1}^{\infty }r_{j}^{n-\alpha p}C_{\alpha
,p,\infty }(\mathbb{B}^{n}).
\end{equation*}%
Therefore, by the definition of $H^{n-\alpha p}(E),$ we have 
\begin{equation*}
C_{\alpha ,p,\infty }(E)\leq \left( \frac{\pi ^{\frac{n-\alpha p}{2}}}{%
\Gamma (1+\frac{n-\alpha p}{2})}\right) ^{-1}H^{n-\alpha p}(E)C_{\alpha
,p,\infty }(\mathbb{B}^{n}).
\end{equation*}

If $p\geq \frac{n}{\alpha }$, then in view of Proposition \ref{p21} and the
ball-decomposition of open sets in $\mathbb{R}^{n}$, it is enough to verify
the result for balls, $B(x_{0},R_{0}).$ We shall consider two cases: Suppose
first that $p>\frac{n}{\alpha },$ then, as $R_{0}\rightarrow \infty ,$ we
have 
\begin{equation*}
C_{\alpha ,p,\infty }\big(B(x_{0},R_{0})\big)=R_{0}^{n-\alpha p}C_{\alpha
,p,\infty }(\mathbb{B}^{n})\rightarrow 0.
\end{equation*}%
Since $C_{\alpha ,p,\infty }$ is monotone, we see that 
\begin{equation*}
C_{\alpha ,p,\infty }\big(B(x_{0},r_{0})\big)\leq C_{\alpha ,p,\infty }\big(%
B(x_{0},R_{0})\rightarrow 0\ \hbox{as}\ R_{0}\rightarrow \infty ,
\end{equation*}%
whence 
\begin{equation*}
C_{\alpha ,p,\infty }\big(B(x_{0},r_{0})\big)=0.
\end{equation*}

Suppose now that $p=\frac{n}{\alpha }$. We have, 
\begin{equation*}
C_{\alpha ,p,\infty }\big(B(x_{0},R_{0})\big)=C_{\alpha ,p,\infty }(\mathbb{B%
}^{n})\ \forall \ R_{0}>0.
\end{equation*}%
By Proposition \ref{p22}(5) we have 
\begin{equation*}
\lim_{R_{0}\rightarrow 0}C_{\alpha ,p,\infty }\big(B(x_{0},R_{0})\big)=0,
\end{equation*}%
whence 
\begin{equation*}
C_{\alpha ,p,\infty }\big(B(x_{0},r_{0})\big)=C_{\alpha ,p,\infty }(\mathbb{B%
}^{n})=0\ \ \forall \ \ r_{0}>0.
\end{equation*}
\end{proof}

\begin{cor}
\label{c21} Let $K$ be a compact subset of $\mathbb{R}^{n}$ and let $%
\partial K$ be its boundary. Let $\mathsf{O}(K)$ stand for the class of all
open sets $O\subset \mathbb{R}^{n}$ with $O\supset K$.

\begin{enumerate}
\item If $(\alpha,p)\in(0,1)\times[1,\frac{n}{\alpha})$, then 
\begin{equation*}
C_{\alpha,p,\infty}(K)=C_{\alpha,p,\infty}(\partial K).
\end{equation*}

\item If $(\alpha ,p)\in (0,1)\times \lbrack 1,\frac{1}{\alpha })$, then 
\begin{equation*}
C_{\alpha ,p,\infty }(K)\leq \inf_{O\in \mathsf{O}(K)}\big(P_{\alpha
,p,\infty }(O)\big)^{p},
\end{equation*}%
with equality when $(\alpha ,p)\in (0,1)\times \{1\}$, i.e., 
\begin{equation*}
C_{\alpha,1,\infty }(K)=\inf_{O\in \mathsf{O}(K)}P_{\alpha ,1,\infty }(O).
\end{equation*}
\end{enumerate}
\end{cor}

\begin{proof}
(1) By Propositions \ref{p22}(2)\&\ref{p23}(5), it is enough to prove 
\begin{equation*}
C_{\alpha ,p,\infty }(K)\leq C_{\alpha ,p,\infty }(\partial K).
\end{equation*}%
Given $f\in \mathsf{A}_{\alpha ,p,\infty }(\partial K)$ let us define 
\begin{equation*}
g=%
\begin{cases}
\max \{f,1\}\quad \hbox{on}\quad K; \\ 
f\quad \hbox{on}\quad \mathbb{R}^{n}\setminus K.%
\end{cases}%
\end{equation*}%
Then, 
\begin{equation*}
g\in \mathsf{A}_{\alpha ,p,\infty }(K)\ \ \&\ \ \Vert g\Vert _{\dot{\Lambda}%
_{\alpha }^{p,\infty }}^{p}\leq \Vert f\Vert _{\dot{\Lambda}_{\alpha
}^{p,\infty }}^{p}.
\end{equation*}%
Combining this fact with the definition of $C_{\alpha ,p,\infty }(K)$,
readily yields 
\begin{equation*}
C_{p,\alpha ,\infty }(K)\leq \Vert f\Vert _{\dot{\Lambda}_{\alpha
}^{p,\infty }}^{p},
\end{equation*}%
and the result follows.

(2) The desired inequality follows from Propositions \ref{p22}(2) \& \ref%
{p23}(2). Now, suppose that $p=1$ and $f\in \mathsf{A}_{\alpha ,1,\infty
}(K) $. Note that, if $t\in (0,1),$ then the upper-level set $\{x\in \mathbb{%
R}^{n}:f(x)>t\},$ is an open set containing $K$. Given $h\in \mathbb{R}^{n},$
and $\epsilon >0$, there exists $\tilde{O}\in \mathsf{O}(K)$ such that 
\begin{equation*}
\inf_{O\in \mathsf{O}(K)}\Vert 1_{O}(\cdot +h)-1_{O}(\cdot )\Vert
_{L^{1}}\geq \Vert 1_{\tilde{O}}(\cdot +h)-1_{\tilde{O}}(\cdot )\Vert
_{L^{1}}-\epsilon .
\end{equation*}%
Therefore, by Fubini's theorem, it follows that 
\begin{align*}
\Vert f\Vert _{\dot{\Lambda}_{\alpha }^{1,\infty }}& =\sup_{h\in \mathbb{R}%
^{n}}|h|^{-\alpha }\Vert \Delta _{h}f\Vert _{L^{1}} \\
& =\sup_{h\in \mathbb{R}^{n}}\int_{0}^{\infty }|h|^{-\alpha }\Vert
1_{\{f>t\}}(\cdot +h)-1_{\{f>t\}}(\cdot )\Vert _{L^{1}}\,dt \\
& \geq \sup_{h\in \mathbb{R}^{n}}\int_{0}^{1}|h|^{-\alpha }\Vert
1_{\{f>t\}}(\cdot +h)-1_{\{f>t\}}(\cdot )\Vert _{L^{1}}\,dt \\
& \geq \sup_{h\in \mathbb{R}^{n}}\int_{0}^{1}|h|^{-\alpha }\inf_{O\in 
\mathsf{O}(K)}\Vert 1_{O}(\cdot +h)-1_{O}(\cdot )\Vert _{L^{1}}\,dt \\
& =\sup_{h\in \mathbb{R}^{n}}|h|^{-\alpha }\inf_{O\in \mathsf{O}(K)}\Vert
1_{O}(\cdot +h)-1_{O}(\cdot )\Vert _{L^{1}} \\
& \geq \sup_{h\in \mathbb{R}^{n}}|h|^{-\alpha }\Big(\Vert 1_{\tilde{O}%
}(\cdot +h)-1_{\tilde{O}}(\cdot )\Vert _{L^{1}}-\epsilon \Big) \\
& \geq \sup_{h\in \mathbb{R}^{n}}|h|^{-\alpha }\Vert 1_{\tilde{O}}(\cdot
+h)-1_{\tilde{O}}(\cdot )\Vert _{L^{1}}-\epsilon \inf_{h\in \mathbb{R}%
^{n}}|h|^{-\alpha } \\
& =P_{\alpha ,1,\infty }(\tilde{O}) \\
& \geq \inf_{O\in \mathsf{O}(K)}P_{\alpha ,1,\infty }(O).
\end{align*}%
As a consequence, we find 
\begin{equation*}
C_{\alpha ,1,\infty }(K)=\inf_{O\in \mathsf{O}(K)}P_{\alpha ,1,\infty }(O).
\end{equation*}
\end{proof}

\section{$\infty$-Besov restrictions, extensions and multipliers}

\label{s3}

\subsection{$\infty$-Besov restrictions}

\label{s31}

Proposition \ref{p23} tells us that $C_{\alpha ,p,\infty }$ is interesting
only when $1\leq p<\frac{n}{\alpha }$. Moreover, according to Proposition %
\ref{p22} this capacity is an outer measure. Therefore, it is a good fit for
the so-called Lebesgue theory for outer measures developed recently in \cite%
{DT}. We can thus try to measure the trace/restriction of a $\Lambda
_{\alpha }^{p,\infty }$-function on a given compact exceptional set via
looking for an outer measure $\mu $ concentrated on this compact set such
that $\Lambda _{\alpha }^{p,\infty }$ embeds continuously into a weak $\mu $%
-based Lorentz space. More precisely, we have the following
trace/restriction result.

\begin{prop}
\label{p31} Let $(\alpha ,p,q)\in (0,1)\times \lbrack 1,\frac{n}{\alpha }%
)\times \lbrack 1,\infty ),$ and let $\mu $ be a nonnegative outer measure
on $\mathbb{R}^{n}$. Consider the following statements:

\begin{enumerate}
\item $\sup_{t\in (0,\infty)}\Big(t^q\mu\big(\{x\in\mathbb{R}^n: |f(x)|>t\}%
\big)\Big)^\frac1q\lesssim \|f\|_{\dot{\Lambda}_\alpha^{p,\infty}}\ \
\forall\ \ f\in\Lambda_{\alpha}^{p,\infty}\cap C(\mathbb{R}^n). $

\item $\mu(E)\lesssim\big(C_{\alpha,p,\infty}(E)\big)^\frac{q}{p}\ \
\forall\ \ \hbox{Borel\ set}\ \ E\subset\mathbb{R}^n. $

\item $\sup_{(t,x,r)\in (0,\infty)\times\mathbb{R}^n\times(0,\infty)}\Big(%
t^q\mu\big(\{y\in B(x,r): |f(y)|>t\}\big)\Big)^\frac1q\lesssim \|f\|_{\dot{%
\Lambda}_\alpha^{p,\infty}}\ \ \forall\ \
f\in\Lambda_{\alpha}^{p,\infty}\cap C(\mathbb{R}^n). $

\item $\mu \big(B(x,r)\big)\lesssim r^{\frac{q(n-\alpha p)}{p}}\ \ \forall \
\ \hbox{Euclidean\ ball}\ \ B(x,r)\subset \mathbb{R}^{n}.$
\end{enumerate}

Then, the following equivalences hold: $(1)\Leftrightarrow (2)$ and $%
(3)\Leftrightarrow (4)$.
\end{prop}

\begin{proof}
$(1)\Leftrightarrow (2)$ Suppose that (1) holds. Let $f\in \mathsf{A}%
_{\alpha ,p,\infty }(E),$ and let $E\subset \mathbb{R}^{n}$ be a Borel set.
Then, 
\begin{equation*}
t^{q}\mu (E)\leq t^{q}\mu \big(\{x\in \mathbb{R}^{n}:|f(x)|>t\}\big)\lesssim
\Vert f\Vert _{\dot{\Lambda}_{\alpha }^{p,\infty }}^{q}\ \ \forall \ \ t\in
(0,1).
\end{equation*}%
Letting $t$ tend to $1,$ and using the definition of $C_{\alpha ,p,\infty
}(E)$ readily yields (2). Conversely, if (2) holds true, then an application
of Proposition \ref{p23}(1) gives 
\begin{equation*}
t^{q}\mu \big(\{x\in \mathbb{R}^{n}:|f(x)|>t\}\big)\lesssim \Big(%
t^{p}C_{\alpha ,p,\infty }\big(\{x\in \mathbb{R}^{n}:|f(x)|>t\}\big)\Big)^{%
\frac{q}{p}}\lesssim \Vert f\Vert _{\dot{\Lambda}_{\alpha }^{p,\infty
}}^{q}\,,
\end{equation*}%
whence (1).

$(3)\Leftrightarrow (4)$ Suppose that (3) is valid, let $f\in \mathsf{A}%
_{\alpha ,p,\infty }\big(B(x,r)\big),$ where $B(x,r)$ is an Euclidean ball.
Then, 
\begin{equation*}
t^{q}\mu \big(B(x,r)\big)\leq t^{q}\mu \big(\{y\in B(x,r):|f(y)|>t\}\big)%
\lesssim \Vert f\Vert _{\dot{\Lambda}_{\alpha }^{p,\infty }}^{q}\ \ \forall
\ \ t\in (0,1).
\end{equation*}%
Letting $t\rightarrow 1,$ and using the definition of $C_{\alpha ,p,\infty }%
\big(B(x,r)\big),$ as well as Proposition \ref{p23}(4), yields 
\begin{equation*}
\mu \big(B(x,r)\big)\lesssim \Big(C_{\alpha ,p,\infty }\big(B(x,r)\big)\Big)%
^{\frac{q}{p}}\approx r^{\frac{q(n-\alpha p)}{p}},
\end{equation*}%
whence (4). Conversely, suppose that (4) is true. Applying Proposition \ref%
{p23}(1) we find that 
\begin{equation*}
t^{q}\mu \big(\{y\in B(x,r):|f(x)|>t\}\big)\lesssim \Big(t^{p}C_{\alpha
,p,\infty }\big(\{y\in B(x,r):|f(y)|>t\}\big)\Big)^{\frac{q}{p}}\lesssim
\Vert f\Vert _{\dot{\Lambda}_{\alpha }^{p,\infty }}^{q}\,
\end{equation*}%
holds for any Euclidean ball $B(x,r)\subset \mathbb{R}^{n}$, whence (3)
holds.
\end{proof}

\begin{rem}
\label{r31} Three comments are in order.

\begin{itemize}
\item[(i)] If $\|f\|_{L^{q,\infty}_\mu}=\sup_{t\in (0,\infty)}\Big(t^q\mu%
\big(\{x\in\mathbb{R}^n: |f(x)|>t\}\big)\Big)^\frac1q, $ then Lyapunov's
inequality (cf. \cite[Proposition 5.3]{CNF}) is: 
\begin{equation*}
\|f\|_{L^{q,\infty}_\mu}\le
\|f\|^{1-\theta}_{L^{q_0,\infty}_\mu}\|f\|^{\theta}_{L^{q_1,\infty}_\mu}%
\quad\forall\quad \frac1{q}=\frac{1-\theta}{q_0}+\frac{\theta}{q_1}\ \ \&\ \
(\theta,q_0,q_1)\in (0,1)\times[1,\infty)\times[1,\infty).
\end{equation*}
Similarly, one has: 
\begin{equation*}
\|f\|_{\dot{\Lambda}^{p,\infty}_\alpha}\le \|f\|^{1-\theta}_{\dot{\Lambda}%
_\alpha^{p_0,\infty}}\|f\|^{\theta}_{\dot{\Lambda}_\alpha^{p_1,\infty}}\quad%
\forall\quad \frac1{p}=\frac{1-\theta}{p_0}+\frac{\theta}{p_1}\ \ \&\ \
(\theta,p_0,p_1)\in (0,1)\times[1,\infty)\times[1,\infty).
\end{equation*}

\item[(ii)] If we let $\mu $ be the $n$-dimensional Lebesgue measure in
Proposition \ref{p31}(1) and we have $n-\alpha p>0$, then the limiting case
of \cite[Theorem 2.8]{KL} or \cite[Corollary 3.3]{S} ensures 
\begin{equation*}
\Vert f\Vert _{L_{\mu }^{\frac{pn}{n-\alpha p},\infty }}\leq 2^{\frac{1}{p}}%
\Big(\frac{pn}{n-\alpha p}\Big)\Vert f\Vert _{\dot{\Lambda}_{\alpha
}^{p,\infty }}\quad \forall \quad f\in \Lambda _{\alpha }^{p,\infty }\cap C(%
\mathbb{R}^n).
\end{equation*}%
Therefore, Proposition \ref{p31}(1) is valid for $q=\frac{pn}{n-\alpha p}$
and the $n$-dimensional Lebesgue measure. Consequently, one has the
following isocapacitary - isoperimetric inequality 
\begin{equation*}
|E|^{\frac{n-\alpha p}{pn}}\leq 2^{\frac{1}{p}}\Big(\frac{pn}{n-\alpha p}%
\Big)\big(C_{\alpha ,p,\infty }(E)\big)^{\frac{1}{p}}\leq 2^{\frac{1}{p}}%
\Big(\frac{pn}{n-\alpha p}\Big)P_{\alpha ,p,\infty }(E)\quad \forall \quad %
\hbox{Borel\ set}\quad E\subset \mathbb{R}^{n}.
\end{equation*}%
We should also remark that, due to Corollary \ref{c21}, the important case
here corresponds to $p\in \lbrack 1,\alpha ^{-1}]$.

\item[(iii)] Proposition \ref{p31}(2) implies Proposition \ref{p31}(4), but
the converse does not necessarily hold.
\end{itemize}
\end{rem}

\begin{cor}
\label{c31} Let $(\alpha ,p,q)\in (0,1)\times \lbrack 1,\frac{n}{\alpha }%
)\times \lbrack 1,\infty )$, let $\mu $ be a nonnegative outer measure on $%
\mathbb{R}^{n}$, and let $\phi :\mathbb{R}^{n}\rightarrow \mathbb{R}^{n}$ be
a Borel map. Referring to the numbered statements below the following
equivalences hold true: $(1)\Leftrightarrow (2)$ and $(3)\Leftrightarrow
(4). $

\begin{enumerate}
\item $\|f\circ\phi\|_{L^{q,\infty}_\mu}\lesssim \|f\|_{\dot{\Lambda}%
_\alpha^{p,\infty}}\ \ \forall\ \ f\in\Lambda_{\alpha}^{p,\infty}\cap C(%
\mathbb{R}^n). $

\item $\mu\big(\phi^{-1}(E)\big)\lesssim \big(C_{\alpha,p,\infty}(E)\big)^%
\frac{q}{p}\ \forall\ \hbox{Borel\ set}\ E\subset\mathbb{R}^n$.

\item $\sup_{(t,x,r)\in (0,\infty)\times\mathbb{R}^n\times(0,\infty)}\Big(%
t^q\mu\big(\{y\in\phi^{-1} \big(B(x,r)\big): |f\circ\phi(y)|>t\}\big)\Big)%
^\frac1q\lesssim \|f\|_{\dot{\Lambda}_\alpha^{p,\infty}}\ \ \forall\ \
f\in\Lambda_{\alpha}^{p,\infty}\cap C(\mathbb{R}^n). $

\item $\mu\big(\phi^{-1}\big(B(x,r))\big)\lesssim r^\frac{q(n-\alpha p)}{p}\
\forall\ \hbox{Euclidean\ ball}\ B(x,r)\subset\mathbb{R}^n$.
\end{enumerate}
\end{cor}

\begin{proof}
Clearly, the above conclusions follow by means of applying Proposition \ref%
{p31} to the push-forward outer measure $\phi _{\ast }\mu $ defined by: 
\begin{equation*}
\phi _{\ast }\mu (E)=\mu \big(\phi ^{-1}(E)\big)\quad \forall \quad %
\hbox{Borel\ set}\ \ E\subset \mathbb{R}^{n}.
\end{equation*}
\end{proof}

\subsection{$\infty$-Besov extensions}

\label{s32} Lifting an arbitrary ${\Lambda }_{\alpha }^{p,\infty }$-function
to the upper-half space $\mathbb{R}_{+}^{1+n}=(0,\infty )\times \mathbb{R}%
^{n}$ via the heat equation, we obtain the Carleson imbedding for $\dot{%
\Lambda}_{\alpha }^{p,\infty }$ via the heat equation (which can be also
generalized to the fractional case; see \cite{XZ}).

\begin{prop}
\label{p32} Let $(\alpha ,p,q)\in (0,1)\times (1,\frac{n}{\alpha })\times
\lbrack 1,\infty )$ and $\nu $ be a nonnegative outer measure on $\mathbb{R}%
_{+}^{1+n}$. Let 
\begin{equation*}
w(t,x)=(4\pi t)^{-\frac{n}{2}}\int_{\mathbb{R}^{n}}\exp \Big(-\frac{|x-y|^{2}%
}{4t}\Big)f(y)\,dy,
\end{equation*}%
be the solution to the heat equation 
\begin{equation*}
\begin{cases}
(\partial _{t}-\Delta _{x})w(t,x)=0\quad \forall \quad (t,x)\in \mathbb{R}%
_{+}^{1+n}; \\ 
w(0,x)=f(x)\quad x\in \mathbb{R}^{n}.%
\end{cases}%
\end{equation*}%
For an open set $O\subset \mathbb{R}^{n},$ let $T(O)=\{(t,x)\in \mathbb{R}%
_{+}^{1+n}:B(x,t)\subseteq O\},$ be the tent with base $O\subset \mathbb{R}%
^{n}$. Then,

\begin{align*}
&\sup_{\lambda\in (0,\infty)}\Big(\lambda^q\nu\big(\{(t,x)\in\mathbb{R}%
^{1+n}_+: |w(t^2,x)|>\lambda\}\big)\Big)^\frac1q\lesssim \|f\|_{\dot{\Lambda}%
_\alpha^{p,\infty}}\ \ \forall\ \ f\in\Lambda_{\alpha}^{p,\infty}\cap C(%
\mathbb{R}^n) \\
&\Longleftrightarrow \nu\big(T(O)\big)\lesssim\big(C_{\alpha,p,\infty}(O)%
\big)^\frac{q}{p}\ \ \forall\ \ \hbox{open\ set}\ \ O\subset\mathbb{R}^{n}.
\end{align*}
\end{prop}

\begin{proof}
Suppose that 
\begin{equation*}
\sup_{\lambda \in (0,\infty )}\Big(\lambda ^{q}\nu \big(\{(t,x)\in \mathbb{R}%
_{+}^{1+n}:|w(t^{2},x)|>\lambda \}\big)\Big)^{\frac{1}{q}}\lesssim \Vert
f\Vert _{\dot{\Lambda}_{\alpha }^{p,\infty }}\ \ \forall \ \ f\in \Lambda
_{\alpha }^{p,\infty }\cap C(\mathbb{R}^n)
\end{equation*}%
holds; then, for any given open set $O\subset \mathbb{R}^{n},$ we can pick $%
f\in \mathsf{A}_{\alpha ,p,\infty }(O)$ and find a dimensional constant $c>0$
such that 
\begin{equation*}
w(t^{2},x)\geq (4\pi t)^{-\frac{n}{2}}\int_{B(x,t)}\exp \Big(-\frac{|x-y|^{2}%
}{4t}\Big)f(y)\,dy\geq c\quad \forall \quad (t,x)\in T(O).
\end{equation*}%
Then, by the definition of $C_{\alpha ,p,\infty }(O)$, we get 
\begin{equation*}
\nu \big(T(O)\big)\lesssim \big(C_{\alpha ,p,\infty }(O)\big)^{\frac{q}{p}}.
\end{equation*}

On the other hand, suppose that the last estimate holds true for any open
set $O\subset \mathbb{R}^{n}$. Then, if $f\in \Lambda _{\alpha }^{p,\infty
}\cap C(\mathbb{R}^n)$, we let 
\begin{equation*}
\mathcal{M}_{N}f(x)=\sup_{|y-x|<t}|w(t^{2},y)|
\end{equation*}%
i.e. the non-tangential maximal function of $w(t^{2},y)$. Since $\mathcal{M}%
_{N}f$ is lower semi-continuous, the level sets $\{x\in \mathbb{R}^{n}:%
\mathcal{M}_{N}f(x)>\lambda \}$ are open for all $\lambda >0$. Moreover,
(cf. \cite{Jo}) 
\begin{equation*}
|w(t^{2},x-y)|\lesssim \Big(1+(|y|t^{-1})^{2}\Big)\mathcal{M}f(x),
\end{equation*}%
where $\mathcal{M}f$ is the standard Hardy-Littlewood maximal function of $f$%
: 
\begin{equation*}
\mathcal{M}f(x)=\sup_{r>0}|B(x,r)|^{-1}\int_{B(x,r)}|f(y)|,dy.
\end{equation*}%
Therefore, there exists a constant $c_{0}>0$ such that $\mathcal{M}_{N}f\leq
c_{0}\mathcal{M}f$. Consequently, 
\begin{align*}
\nu \big(\{(t,x)\in \mathbb{R}_{+}^{1+n}:|w(t^{2},x)|>\lambda \}\big)& \leq
\nu \Big(T\big(\{x\in \mathbb{R}^{n}:\mathcal{M}_{N}f(x)>\lambda \}\big)\Big)
\\
& \leq \nu \Big(T\big(\{x\in \mathbb{R}^{n}:\mathcal{M}f(x)>c_{0}\lambda \}%
\big)\Big) \\
& \lesssim \Big(C_{\alpha ,p,\infty }\big(\{x\in \mathbb{R}^{n}:\mathcal{M}%
f(x)>c_{0}\lambda \}\big)\Big)^{\frac{q}{p}}.
\end{align*}%
Since $\mathcal{M}$ is bounded on $L^{p}$ we see that (cf. \cite[p.3247]{W}) 
\begin{equation*}
\Vert \Delta _{h}(\mathcal{M}f)\Vert _{L^{p}}\lesssim \Vert \mathcal{M}%
(\Delta _{h}f)\Vert _{L^{p}}\lesssim \Vert \Delta _{h}f\Vert _{L^{p}}\quad
\forall \quad (p,h)\in (1,\infty )\times \mathbb{R}^{n}.
\end{equation*}%
Using Proposition \ref{p22}(1), we get 
\begin{equation*}
\Big(\lambda ^{p}C_{\alpha ,p,\infty }\big(\{x\in \mathbb{R}^{n}:\mathcal{M}%
f(x)>c_{0}\lambda \}\big)\Big)^{\frac{q}{p}}\lesssim \Vert \mathcal{M}f\Vert
_{\dot{\Lambda}_{\alpha }^{p,\infty }}^{q}\lesssim \Vert f\Vert _{\dot{%
\Lambda}_{\alpha }^{p,\infty }}^{q},
\end{equation*}%
whence 
\begin{equation*}
\lambda ^{q}\nu \big(\{(t,x)\in \mathbb{R}_{+}^{1+n}:|w(t^{2},x)|>\lambda \}%
\big)\lesssim \Vert f\Vert _{\dot{\Lambda}_{\alpha }^{p,\infty }}^{q}.
\end{equation*}
\end{proof}

\subsection{$\infty$-Besov point-wise multipliers}

In view of Proposition \ref{p31}, we can naturally deal with the following
multiplication problem for $\dot{\Lambda}_{\alpha }^{p,\infty }$.

\begin{prop}
\label{p41} Let $(\alpha ,p,q)\in (0,1)\times \lbrack 1,\frac{n}{\alpha }%
)\times \lbrack 1,\infty )$, let $\mu $ be a nonnegative outer measure on $%
\mathbb{R}^{n}$, and let $\mathsf{m}:\mathbb{R}^{n}\rightarrow \mathbb{R}$
be a Borel function. Then, one has the following implications:

\begin{enumerate}
\item 
\begin{align*}
&\mathsf{m}\in L^\infty_\mu\ \ \&\ \ \mu(E)\lesssim \big(C_{\alpha,p,%
\infty}(E)\big)^\frac{q}{p}\ \forall\ \hbox{Borel\ set}\ E\subset\mathbb{R}^n
\\
&\Longrightarrow \|\mathsf{m}f\|_{L^{q,\infty}_\mu}\lesssim \|f\|_{\dot{%
\Lambda}_\alpha^{p,\infty}}\ \ \forall\ \
f\in\Lambda_{\alpha}^{p,\infty}\cap C(\mathbb{R}^n) \\
&\Longrightarrow\sup_{t\in (0,\infty)}t^q\mu\big(\{x\in E:\ |\mathsf{m}%
(x)|>t\}\big)\lesssim \big(C_{\alpha,p,\infty}(E)\big)^\frac{q}{p}\ \forall\ %
\hbox{Borel\ set}\ E\subset\mathbb{R}^n.
\end{align*}

\item 
\begin{align*}
&\mathsf{m}\in L^\infty_\mu\ \ \&\ \ \mu\big(B(x,r)\big)\lesssim r^\frac{%
q(n-\alpha p)}{p}\ \forall\ \hbox{Euclidean\ ball}\ B(x,r)\subset\mathbb{R}^n
\\
&\Longrightarrow \sup_{(t,x,r)\in (0,\infty)\times\mathbb{R}%
^n\times(0,\infty)}\Big(t^q\mu\big(\{y\in B(x,r): |\mathsf{m}(y)f(y)|>t\}%
\big)\Big)^\frac1q\lesssim \|f\|_{\dot{\Lambda}_\alpha^{p,\infty}}\ \
\forall\ \ f\in\Lambda_{\alpha}^{p,\infty}\cap C(\mathbb{R}^n) \\
&\Longrightarrow \sup_{t\in (0,\infty)}t^q\mu\big(\{y\in B(x,r):\ |\mathsf{m}%
(y)|>t\}\big)\lesssim r^\frac{q(n-\alpha p)}{p}\ \forall\ 
\hbox{Euclidean\
ball}\ B(x,r)\subset\mathbb{R}^n.
\end{align*}
\end{enumerate}
\end{prop}

\begin{proof}
It is enough to check Proposition \ref{p41}(1). Suppose that 
\begin{equation*}
\mathsf{m}\in L_{\mu }^{\infty }\ \ \&\ \ \mu (E)\lesssim \big(C_{\alpha
,p,\infty }(E)\big)^{\frac{q}{p}}\ \forall \ \hbox{Borel\ set}\ E\subset 
\mathbb{R}^{n}.
\end{equation*}%
Without loss of generality we may assume that $\Vert \mathsf{m}\Vert
_{L_{\mu }^{\infty }}>0$. Note that if $f\in \Lambda _{\alpha }^{p,\infty
}\cap C(\mathbb{R}^n),$ then 
\begin{equation*}
t<|\mathsf{m}(x)f(x)|\leq \Vert \mathsf{m}\Vert _{L_{\mu }^{\infty
}}|f(x)|\Rightarrow t\Vert \mathsf{m}\Vert _{L_{\mu }^{\infty }}^{-1}<|f(x)|.
\end{equation*}%
Therefore, an application of Proposition \ref{p31}(1)$\Leftrightarrow $(2),
yields 
\begin{equation*}
\Vert \mathsf{m}\Vert _{L_{\mu }^{\infty }}^{q}\Vert f\Vert _{\Lambda
_{\alpha }^{p,\infty }}^{q}\gtrsim \Vert \mathsf{m}\Vert _{L_{\mu }^{\infty
}}^{q}\left( \frac{t}{\Vert \mathsf{m}\Vert _{L_{\mu }^{\infty }}}\right)
^{q}\mu \Big(\big\{x\in \mathbb{R}^{n}:\ |f(x)|>\frac{t}{\Vert \mathsf{m}%
\Vert _{L_{\mu }^{\infty }}}\big\}\Big)\geq \Vert \mathsf{m}f\Vert _{L_{\mu
}^{q,\infty }}^{q}.
\end{equation*}

Next, we assume 
\begin{equation*}
\Vert \mathsf{m}f\Vert _{L_{\mu }^{q,\infty }}\lesssim \Vert f\Vert _{\dot{%
\Lambda}_{\alpha }^{p,\infty }}\ \ \forall \ \ f\in \Lambda _{\alpha
}^{p,\infty }\cap C(\mathbb{R}^n),
\end{equation*}
Let $E$ by a Borel set in $\mathbb{R}^{n}$, and let $g\in \mathsf{A}%
_{p,\infty ,\alpha }(E)$. Then, 
\begin{equation*}
\{x\in E:\ |\mathsf{m}(x)g(x)|>t\}\supseteq \{x\in E:\ |\mathsf{m}%
(x)|>t\}\quad \forall \quad t\in (0,\infty ).
\end{equation*}%
Whence, 
\begin{align*}
\Vert g\Vert _{\dot{\Lambda}_{\alpha }^{p,\infty }}^{q}& \gtrsim \Vert 
\mathsf{m}g\Vert _{L_{\mu }^{q,\infty }}^{q} \\
& \gtrsim \sup_{t\in (0,\infty )}t^{q}\mu \big(\{x\in E:\ |\mathsf{m}%
(x)g(x)|>t\}\big) \\
& \gtrsim \sup_{t\in (0,\infty )}t^{q}\mu \big(\{x\in E:\ |\mathsf{m}(x)|>t\}%
\big).
\end{align*}%
Taking the infimum over all such $g$ finally yields 
\begin{equation*}
\big(C_{\alpha ,p,\infty }(E)\big)^{\frac{q}{p}}\gtrsim \sup_{t\in (0,\infty
)}t^{q}\mu \big(\{x\in E:\ |\mathsf{m}(x)|>t\}\big).
\end{equation*}
\end{proof}

\begin{rem}
\label{r31} We ask wether the following implications 
\begin{align*}
& \sup_{t\in (0,\infty )}t^{q}\mu \big(\{x\in E:\ |\mathsf{m}(x)|>t\}\big)%
\lesssim \Vert \mathsf{m}\Vert _{L_{\mu }^{\infty }}^{q}\big(C_{\alpha
,p,\infty }(E)\big)^{\frac{q}{p}}\ \forall \ \hbox{Borel\ set}\ E\subset 
\mathbb{R}^{n} \\
& \Longrightarrow \mathsf{m}\in L_{\mu }^{\infty }\ \ \&\ \ \mu (E)\lesssim %
\big(C_{\alpha ,p,\infty }(E)\big)^{\frac{q}{p}}\ \forall \ \hbox{Borel\ set}%
\ E\subset \mathbb{R}^{n}
\end{align*}%
and 
\begin{align*}
& \sup_{t\in (0,\infty )}t^{q}\mu \big(\{y\in B(x,r):\ |\mathsf{m}(y)|>t\}%
\big)\lesssim r^{\frac{q(n-\alpha p)}{p}}\ \forall \ \hbox{Euclidean\ ball}\
B(x,r)\subset \mathbb{R}^{n} \\
& \Longrightarrow \mathsf{m}\in L_{\mu }^{\infty }\ \ \&\ \ \mu \big(B(x,r)%
\big)\lesssim r^{\frac{q(n-\alpha p)}{p}}\ \forall \ \hbox{Euclidean\ ball}\
B(x,r)\subset \mathbb{R}^{n}
\end{align*}%
are true or not? Obviously, if the essential lower bound of the $L_{\mu
}^{\infty }$-function $\mathsf{m}$ is positive, then the answer to the above
questions is affirmative.

Moreover, if $\mu $ is a nonnegative Radon measure, then \cite[Proposition
5.1]{CNF} implies 
\begin{align*}
& \sup_{t\in (0,\infty )}t^{q}\mu \big(\{x\in E:\ |\mathsf{m}(x)|>t\}\big) \\
& \leq \sup_{0<\mu (E)<\infty }\mu (E)\left( \int_{E}|\mathsf{m}|^{r}\,\frac{%
d\mu }{\mu (E)}\right) ^{\frac{q}{r}} \\
& \leq \left( \frac{q}{q-r}\right) ^{\frac{q}{r}}\sup_{t\in (0,\infty
)}t^{q}\mu \big(\{x\in E:\ |\mathsf{m}(x)|>t\}\big)\quad \forall \quad r\in
(0,q).
\end{align*}%
Letting $r\rightarrow 0,$ and using \cite[Exercise 6.117 (b)]{G}, we get 
\begin{align*}
& \sup_{t\in (0,\infty )}t^{q}\mu \big(\{x\in E:\ |\mathsf{m}(x)|>t\}\big) \\
& \leq \sup_{0<\mu (E)<\infty }\mu (E)\exp \left( \int_{E}\log |\mathsf{m}%
|^{q}\,\frac{d\mu }{\mu (E)}\right) \\
& \leq e^{1/p}\sup_{t\in (0,\infty )}t^{q}\mu \big(\{x\in E:\ |\mathsf{m}%
(x)|>t\}\big).
\end{align*}%
Thus, 
\begin{align*}
& \sup_{t\in (0,\infty )}t^{q}\mu \big(\{x\in E:\ |\mathsf{m}(x)|>t\}\big)%
\lesssim \Vert \mathsf{m}\Vert _{L_{\mu }^{\infty }}^{q}\big(C_{\alpha
,p,\infty }(E)\big)^{\frac{q}{p}}\ \forall \ \hbox{Borel\ set}\ E\subset 
\mathbb{R}^{n} \\
& \Longrightarrow \sup_{0<\mu (E)<\infty }\mu (E)\exp \left( \int_{E}\log | 
\mathsf{m}|^{q}\,\frac{d\mu }{\mu (E)}\right) \lesssim \big(C_{\alpha
,p,\infty }(E)\big)^{\frac{q}{p}}.
\end{align*}%
Similarly, one has 
\begin{align*}
& \sup_{t\in (0,\infty )}t^{q}\mu \big(\{y\in B(x,r):\ |\mathsf{m}(y)|>t\}%
\big)\lesssim r^{\frac{q(n-\alpha p)}{p}}\ \forall \ \hbox{Euclidean\ ball}\
B(x,r)\subset \mathbb{R}^{n} \\
& \Longrightarrow \sup_{(x,r)\in \mathbb{R}^{n}\times (0,\infty )}\mu \big(%
B(x,r))\exp \left( \int_{B(x,r)}\log |\mathsf{m}|^{q}\,\frac{d\mu }{\mu \big(%
B(x,r)\big)}\right) \lesssim r^{\frac{q(n-\alpha p)}{p}}.
\end{align*}
\end{rem}

\end{document}